\documentclass[11pt]{amsart}
\usepackage{amssymb}
\usepackage{amsmath}
\usepackage{amsthm}
\usepackage{enumerate}
\usepackage{mathrsfs}
\newtheorem{theorem}{Theorem}[section]
\newtheorem{lemma}[theorem]{Lemma}
\newtheorem{corollary}[theorem]{Corollary}
\theoremstyle{definition}

\newtheorem*{question}{Question}

\numberwithin{equation}{section}
\mathchardef\hyphen="2D
\begin{document}
\allowdisplaybreaks
\title{Embedding $C^{*}$-algebras into the Calkin algebra of $\ell^{p}$}
\author{March T.~Boedihardjo}
\address{Department of Mathematics, Michigan State University, Michigan, MI 48824}
\email{boedihar@msu.edu}
\begin{abstract}
Let $p\in(1,\infty)$. We show that there is an isomorphism from any separable unital subalgebra of $B(\ell^{2})/K(\ell^{2})$ onto a subalgebra of $B(\ell^{p})/K(\ell^{p})$ that preserves the Fredholm index. As a consequence, every separable $C^{*}$-algebra is isomorphic to a subalgebra of $B(\ell^{p})/K(\ell^{p})$. Another consequence is the existence of operators on $\ell^{p}$ that behave like the essentially normal operators with arbitrary Fredholm indices in the Brown-Douglas-Fillmore theory.
\end{abstract}
\maketitle
\section{Introduction}
Let $p\in(1,\infty)\backslash\{2\}$. While the structure of operators on $\ell^{2}$ has been extensively studied in the literature, the structure of operators on $\ell^{p}$ remains mysterious. Perhaps outside a small class of trivial exceptions, there is no resemblance of operators on $\ell^{p}$ to operators on $\ell^{2}$. Or perhaps there is a large class of operators on $\ell^{p}$ that behave like operators on $\ell^{2}$. The literature so far suggests that the former is true. Indeed, the only $C^{*}$-algebras that are isometrically isomorphic to a subalgebra of $B(\ell^{p})$ are the commutative $C^{*}$-algebras \cite{garthi}, where $B(\ell^{p})$ is the algebra of bounded linear operators on $\ell^{p}$. Even if one relaxes isometric isomorphism to isomorphism (i.e., allowing some norm distortion up to a constant), only the residually finite dimensional $C^{*}$-algebras can be embedded into $B(\ell^{p})$ \cite{isorep}. As for single operators, perhaps, the simplest operator on $\ell^{p}$ beyond the diagonal operators is the bilateral shift. Yet, the norm of polynomials of the bilateral shift on $\ell^{p}$ already behaves quite differently from that of the bilateral shift on $\ell^{2}$ which is a normal operator. Indeed, there exist Laurent polynomials $f_{1},f_{2},\ldots$ such that $\sup_{w\in\mathbb{C},\,|w|=1}|f_{n}(w)|=1$, for all $n\in\mathbb{N}$, and $\|f_{n}(B)\|\to\infty$ as $n\to\infty$, where $B$ is the bilateral shift on $\ell^{p}$ \cite{fixman}.

In this paper, we show that the situation is quite different in the context of Calkin algebras. More precisely, we prove that every separable unital subalgebra of $B(\ell^{2})/K(\ell^{2})$ can embedded into a subalgebra of $B(\ell^{p})/K(\ell^{p})$ for which the Fredholm index is preserved, where $K(\ell^{p})$ is the ideal of compact operators on $\ell^{p}$. Thus, every separable $C^{*}$-algebra is isomorphic to a subalgebra of $B(\ell^{p})/K(\ell^{p})$.

In the context of single operator theory, while the theory of spectral operators developed by Dunford, Schwartz and others \cite{ds} provides a class of operators on Banach spaces that behave like the normal operators on Hilbert spaces, it is not even known if there is a class of operators on $\ell^{p}$ that behave like the essentially normal operators with arbitrary Fredholm indices in the Brown-Douglas-Fillmore theory. For example, the unilateral shift on $\ell^{2}$ is an essentially normal operator (with nontrivial Fredholm index) and thus has a functional calculus in the Calkin algebra, but it is not even obvious whether such an operator on $\ell^{p}$ exists.
\begin{question}
Fix $p\in(1,\infty)\backslash\{2\}$. Does there exist a constant $C_{p}\geq 1$ and a Fredholm operator $T$ on $\ell^{p}$ with nontrivial Fredholm index such that
\begin{equation}\label{11}
\|f(\pi(T))\|\leq C_{p}\sup_{w\in\mathbb{C},\,|w|=1}|f(w)|,
\end{equation}
for all Laurent polynomials $f$, where $\pi:B(\ell^{p})\to B(\ell^{p})/K(\ell^{p})$ is the quotient map?
\end{question}
In other words, is there an operator on $\ell^{p}$ that behaves like the unilateral shift on $\ell^{2}$ in terms of the Fredholm index and functional calculus in the Calkin algebra? It turns out that while the obvious candidate, namely, the unilateral shift on $\ell^{p}$ does not satisfy the inequality (\ref{11}) (see \cite{wang}), such an operator on $\ell^{p}$ does exist. In fact, our main result provides a class of operators on $\ell^{p}$ that behave like the essentially normal operators with arbitrary Fredholm indices in the Brown-Douglas-Fillmore theory.

A homomorphism $\phi$ from a Banach algebra $\mathcal{A}$ into another Banach algebra is a bounded linear map for which $\phi(ab)=\phi(a)\phi(b)$ for all $a,b\in\mathcal{A}$. A homomorphism $\phi$ is {\it isomorphic} if there exists $C\geq 1$ such that
\[\frac{1}{C}\|a\|\leq\|\phi(a)\|\leq C\|a\|,\]
for all $a\in\mathcal{A}$.

The following theorem is the main result of this paper.
\begin{theorem}\label{mainthm}
Let $p\in(1,\infty)$. Let $\mathcal{A}$ be a separable unital subalgebra of $B(\ell^{2})/K(\ell^{2})$. Then there exists a unital homomorphism $\phi:\mathcal{A}\to B(\ell^{p})/K(\ell^{p})$ such that
\[\frac{1}{C_{p}}\|a\|\leq\|\phi(a)\|\leq C_{p}\|a\|,\]
for all $a\in\mathcal{A}$, where $C_{p}\geq 1$ is a constant that depends only on $p$. Moreover, $\phi(a)$ and $a$ have the same Fredholm index for every $a\in\mathcal{A}$ that is invertible in $B(\ell^{2})/K(\ell^{2})$.
\end{theorem}
\begin{corollary}
Let $p\in(1,\infty)$. For every separable $C^{*}$-algebra $\mathcal{A}$, there is an isomorphic homomorphism $\phi:\mathcal{A}\to B(\ell^{p})/K(\ell^{p})$.
\end{corollary}
\begin{corollary}
Let $p\in(1,\infty)$. There exists a Fredholm operator $T$ on $\ell^{p}$ with Fredholm index $1$ such that
\[\|f(\pi(T))\|\leq C_{p}\sup_{w\in\mathbb{C},\,|w|=1}|f(w)|,\]
for all Laurent polynomials $f$, where $C_{p}\geq 1$ is a constant that depends only on $p$ and $\pi:B(\ell^{p})\to B(\ell^{p})/K(\ell^{p})$ is the quotient map
\end{corollary}
\begin{proof}
Let $\pi_{2}:B(\ell^{2})\to B(\ell^{2})/K(\ell^{2})$ be the quotient map. Let $U$ and $U^{*}$ be the forward unilateral shift and the backward shift $U^{*}$ on $\ell^{2}$, respectively. In Theorem \ref{mainthm}, take $\mathcal{A}$ to be the subalgebra of $B(\ell^{2})/K(\ell^{2})$ generated by $\pi_{2}(U)$ and $\pi_{2}(U^{*})$. We obtain a unital homomorphism $\phi:\mathcal{A}\to B(\ell^{p})/K(\ell^{p})$ preserving the Fredholm index such that $\|\phi(a)\|\leq C_{p}\|a\|$ for all $a\in\mathcal{A}$. Take any $T\in B(\ell^{p})$ so that $\pi(T)=\phi(\pi_{2}(U^{*}))$. Then since $\phi$ preserves the Fredholm index, $T$ has Fredholm index $1$. Moreover,
\[\|f(\pi(T))\|=\|f(\phi(\pi_{2}(U^{*})))\|=\|\phi(f(\pi_{2}(U^{*})))\|\leq C_{p}\|f(\pi_{2}(U^{*}))\|=C_{p}\sup_{w\in\mathbb{C},\,|w|=1}|f(w)|,\]
for all Laurent polynomials $f$.
\end{proof}
\begin{corollary}
Let $M$ be a compact subset of $\mathbb{C}$ and let $C(M)$ be the algebra of scalar valued continuous functions on $M$. Let $O_{1},O_{2},\ldots$ be the bounded components of $\mathbb{C}\backslash M$. For each $O_{i}$, let $n_{i}\in\mathbb{Z}$. Let $z\in C(M)$ be the identity function. Then there exists a unital isomorphic homomorphism $\psi:C(M)\to B(\ell^{p})/K(\ell^{p})$ such that $\psi(z-\lambda)$ has Fredholm index $n_{i}$ for all $\lambda\in O_{i}$ and all $i$.
\end{corollary}
\begin{proof}
Brown, Douglas and Fillmore \cite{bdf} show that there is a unital isomorphic homomorphism $\phi_{1}:C(M)\to B(\ell^{2})/K(\ell^{2})$ such that $\phi_{1}(z-\lambda)$ has Fredholm index $n_{i}$ for all $\lambda\in O_{i}$ and all $i$. Let $\mathcal{A}$ be the range of $\phi_{1}$. By Theorem \ref{mainthm}, there is a unital isomorphic homomorphism $\phi_{2}:\mathcal{A}\to B(\ell^{p})/K(\ell^{p})$ such that $\phi_{2}(a)$ and $a$ have the same Fredholm index for every $a\in\mathcal{A}$ that is invertible in $B(\ell^{2})/K(\ell^{2})$. The result follows by taking $\psi=\phi_{2}\circ\phi_{1}$.
\end{proof}
We end this section by sketching the proof of the main result Theorem \ref{mainthm}. Let $\mathcal{B}$ be the inverse image of $\mathcal{A}$ in $B(\ell^{2})$. Since $\mathcal{B}$ is separable, all the operators in $\mathcal{B}$ can be simultaneously block tridiagonalized, up to compact perturbations, with respect to some decomposition of $\ell^{2}=\mathcal{H}_{1}\oplus\mathcal{H}_{2}\oplus\ldots$, where $\mathcal{H}_{1},\mathcal{H}_{2},\ldots$ are finite dimensional Hilbert spaces. By block tridiagonality, for each given operator $T\in\mathcal{B}$ on $\ell^{2}$, if we replace the space $\ell^{2}=\mathcal{H}_{1}\oplus\mathcal{H}_{2}\oplus\ldots$ by the $\ell^{p}$ direct sum $X_{p}=(\mathcal{H}_{1}\oplus\mathcal{H}_{2}\oplus\ldots)_{p}$ but keeping the same block matrix representation of $T$, then the essential norm of the induced operator on $X_{p}$ is equivalent to the essential norm of the original operator $T$ up to a constant multiple of $3$. This defines an isomorphic homomorphism from $\mathcal{A}$ into $B(X_{p})/K(X_{p})$. Finally, since the Banach space $X_{p}$ is isomorphic to $\ell^{p}$ by a Pe{\l}czynski decomposition argument, this establishes the isomorphic homomorphism $\phi$ in the conclusion of Theorem \ref{mainthm}. The preservation of the Fredholm index follows from a connectivity argument.
\section{Proof of the main result}
The $\ell^{p}$-direct sum of a sequence $E_{1},E_{2},\ldots$ of Banach spaces is denoted by $(\bigoplus_{k\in\mathbb{N}}E_{k})_{p}$. The essential norm $\|T\|_{e}$ of an operator $T:E\to F$ is the distance between $T$ and the space of all compact operators from $E$ to $F$.

The space $\ell^{2}$ consists of all $\ell^{2}$-summable scalar valued functions on $\mathbb{N}$. If $S\subset\mathbb{N}$, then $\ell^{2}(S)$ consists of all $x\in\ell^{2}$ such that $x(j)=0$ for all $j\notin S$. Note that if $S_{1},S_{2},\ldots$ is a partition of $\mathbb{N}$, we can identify $\ell^{2}$ as $\bigoplus_{k=1}^{\infty}\ell^{2}(S_{k})$.

\begin{lemma}\label{blockbound}
Let $p\in(1,\infty)$ and let $E_{1},E_{2},\ldots$ be finite dimensional Banach spaces. Suppose that we are given an array of operators $T_{i,j}:E_{j}\to E_{i}$, for $i,j\in\mathbb{N}$, such that $\sup_{i,j}\|T_{i,j}\|<\infty$ and $T_{i,j}=0$ whenever $|i-j|>m$ for some fixed $m\in\mathbb{N}$. Then the operator on $(\bigoplus_{k\in\mathbb{N}}E_{k})_{p}$ for which the block matrix representation consists of the array $(T_{i,j})_{i,j\in\mathbb{N}}$ is well defined and has norm between $\sup_{i,j}\|T_{i,j}\|$ and $(2m+1)\sup_{i,j}\|T_{i,j}\|$.
\end{lemma}
\begin{proof}
For each $i\in\mathbb{N}$, let $J_{i}$ be the canonical embedding of $E_{i}$ into $(\bigoplus_{k\in\mathbb{N}}E_{k})_{p}$ and let $Q_{i}$ be the canonical projection from $(\bigoplus_{k\in\mathbb{N}}E_{k})_{p}$ onto $E_{i}$. For notational convenience, let $J_{-i}=0$ and $T_{-i,j}=0$ for all $i\geq 0$ and $j\in\mathbb{N}$. The operator in the statement of this result is given by
\[\sum_{i,j\in\mathbb{N}}J_{i}T_{i,j}Q_{j}=\sum_{s=-m}^{m}\sum_{j=1}^{\infty}J_{j+s}T_{j+s,j}Q_{j},\]
where this sum converges in the strong operator topology. To prove that this sum indeed converges, note that for $s\in\mathbb{Z}$, $k_{1}\leq k_{2}$ in $\mathbb{N}$ and $x\in(\bigoplus_{k\in\mathbb{N}}E_{k})_{p}$, we have
\begin{eqnarray}\label{jtqnorm}
\left\|\sum_{j=k_{1}}^{k_{2}}J_{j+s}T_{j,j+s}Q_{j}x\right\|&=&\left(\sum_{j=k_{1}}^{k_{2}}\|T_{j+s,j}Q_{j}x\|^{p}\right)^{\frac{1}{p}}\\&\leq&
\sup_{j\in\mathbb{N}}\|T_{j+s,j}\|\left(\sum_{j=k_{1}}^{k_{2}}\|Q_{j}x\|^{p}\right)^{\frac{1}{p}},\nonumber
\end{eqnarray}
and since $\sum_{j=1}^{\infty}\|Q_{j}x\|^{p}=\|x\|^{p}$, it follows that $\sum_{j=1}^{\infty}J_{j+s}T_{j,j+s}Q_{j}x$ converges. To estimate the norm of the operator $\sum_{i,j\in\mathbb{N}}J_{i}T_{i,j}Q_{j}$, observe that the norm of this operator is at least the norm of every block $\|J_{i}T_{i,j}Q_{j}\|=\|T_{i,j}\|$. On the other hand,
\begin{eqnarray*}
\left\|\sum_{i,j\in\mathbb{N}}J_{i}T_{i,j}Q_{j}\right\|&=&\left\|\sum_{s=-m}^{m}\sum_{j=1}^{\infty}J_{j+s}T_{j+s,j}Q_{j}\right\|\\&\leq&
\sum_{s=-m}^{m}\left\|\sum_{j=1}^{\infty}J_{j+s}T_{j+s,j}Q_{j}\right\|\\&\leq&
\sum_{s=-m}^{m}\sup_{j\in\mathbb{N}}\|T_{j+s,j}\|\leq(2m+1)\sup_{i,j\in\mathbb{N}}\|T_{i,j}\|,
\end{eqnarray*}
where the second inequality follows from (\ref{jtqnorm}).
\end{proof}
\begin{lemma}\label{mainlemma}
Let $p\in(1,\infty)$ and $0=r_{1}<r_{2}<\ldots$. For each $m\in\mathbb{N}$, let
\[\mathcal{W}_{m}=\{T\in B(\ell^{2}):\,T\ell^{2}([r_{k}+1,r_{k+1}])\subset\bigoplus_{i=k-m}^{k+m}\ell^{2}([r_{i}+1,r_{i+1}])\,\forall k\in\mathbb{N}\},\]
where $r_{-i}=0$ for every $i\geq 0$. Consider the subalgebra $\bigcup_{m\in\mathbb{N}}\mathcal{W}_{m}$ of $B(\ell^{2})$. Let $X_{p}$ be the Banach space $(\bigoplus_{k\in\mathbb{N}}\ell^{2}([r_{k}+1,r_{k+1}]))_{p}$. Then there exists a unital homomorphism $\Psi:\bigcup_{m\in\mathbb{N}}\mathcal{W}_{m}\to B(X_{p})$ such that
\[\frac{1}{2m+1}\|T\|\leq\|\Psi(T)\|\leq(2m+1)\|T\|,\]
and
\[\frac{1}{2m+1}\|T\|_{e}\leq\|\Psi(T)\|_{e}\leq(2m+1)\|T\|_{e},\]
for all $T\in\mathcal{W}_{m}$ and $m\in\mathbb{N}$. Moreover, if $U^{*}$ is the backward unilateral shift on $\ell^{2}$, then $U^{*}\in\mathcal{W}_{1}$ and $\Psi(U^{*})$ has Fredholm index $1$.
\end{lemma}
\begin{proof}
Since every $T\in\mathcal{W}_{m}$ is an operator on $\ell^{2}=(\bigoplus_{k\in\mathbb{N}}\ell^{2}([r_{k}+1,r_{k+1}]))_{2}$, we can define an operator $\Psi(T)$ on $X_{p}=(\bigoplus_{k\in\mathbb{N}}\ell^{2}([r_{k}+1,r_{k+1}]))_{p}$ by keeping the same block matrix representation of $T$ with respect to the decomposition $\bigoplus_{k\in\mathbb{N}}\ell^{2}([r_{k}+1,r_{k+1}])$. In view of Lemma \ref{blockbound}, this operator $\Psi(T)$ is well defined and
\[\sup_{i,j}\|T_{i,j}\|\leq\|\Psi(T)\|\leq(2m+1)\sup_{i,j}\|T_{i,j}\|.\]
And applying Lemma \ref{blockbound} for $p=2$ to the operator $T$ itself, we also have
\[\sup_{i,j}\|T_{i,j}\|\leq\|T\|\leq(2m+1)\sup_{i,j}\|T_{i,j}\|.\]
Therefore,
\begin{equation}\label{Psibound}
\frac{1}{2m+1}\|T\|\leq\|\Psi(T)\|\leq(2m+1)\|T\|,
\end{equation}
for all $T\in\mathcal{W}_{m}$.

Let $\mathcal{F}$ be the set of all $T\in B(\ell^{2})$ with finitely many nonzero entries in its canonical matrix representation. Note that $\mathcal{F}\subset\bigcup_{m\in\mathbb{N}}\mathcal{W}_{m}$. Observe that the norm closure $\mathcal{F}$ coincides with the set of all compact operators on $\ell^{2}$, whereas the norm closure of $\{\Psi(F):F\in\mathcal{F}\}$ coincides with the set of all compact operators on $X_{p}$. Since by (\ref{Psibound}),
\[\frac{1}{2m+1}\inf_{F\in\mathcal{F}}\|T+F\|\leq\inf_{F\in\mathcal{F}}\|\Psi(T)+\Psi(F)\|\leq(2m+1)\inf_{F\in\mathcal{F}}\|T+F\|,\]
it follows that
\[\frac{1}{2m+1}\|T\|_{e}\leq\|\Psi(T)\|_{e}\leq(2m+1)\|T\|_{e},\]
for all $T\in\mathcal{W}_{m}$.

It is obvious that $\Psi$ is a unital homomorphism by identifying all the operators in the domain and range as formal block matrices. This identification is valid since all the $T$ in the domain are in $\mathcal{W}_{m}$ for some $m\in\mathbb{N}$ and thus every sum that arises in computing the product of two formal block matrices is a finite sum.

To prove the moreover statement of this result, it is easy to see from the block matrix representation that both the forward unilateral shift $U$ and the backward unilateral shift $U^{*}$ are in $\mathcal{W}_{1}$. Let $(z_{s})_{s\in\mathbb{N}}$ be the canonical basis for $X_{p}=(\bigoplus_{k\in\mathbb{N}}\ell^{2}([r_{k}+1,r_{k+1}]))_{p}$. Then by the definition of $\Psi$, we have $\Psi(U)z_{s}=z_{s+1}$ and $\Psi(U^{*})z_{s}=z_{s-1}$ for all $s\in\mathbb{N}$, where $z_{0}=0$. It is easy to see that the kernel of the operator $\Psi(U^{*})\in B(X_{p})$ is spanned by $z_{1}$ and thus has dimension $1$. Since $\Psi(U^{*})\Psi(U)=\Psi(U^{*}U)=\Psi(I)=I$, the operator $\Psi(U^{*})$ is surjective. Therefore, the Fredholm index of $\Psi(U^{*})$ is $1$.
\end{proof}
The following result is inspired by the proof of \cite[Lemma 1.2]{NCweylvon}.
\begin{lemma}\label{blocktridiag}
Let $\mathcal{B}$ be a separable subset of $B(\ell^{2})$. Then there are $0=r_{1}<r_{2}<\ldots$ such that $T-\sum_{k=1}^{\infty}(Q_{k-1}+Q_{k}+Q_{k+1})TQ_{k}$ is compact for every $T\in\mathcal{B}$, where $Q_{0}=0$ and $Q_{k}$ is the canonical projection from $\ell^{2}$ onto $\ell^{2}([r_{k}+1,r_{k+1}])$. Here the sum $\sum_{k=1}^{\infty}(Q_{k-1}+Q_{k}+Q_{k+1})TQ_{k}$ is in the strong operator topology.
\end{lemma}
\begin{proof}
Since $\mathcal{B}$ is separable, there exists a dense sequence $T_{1},T_{2},\ldots$ in $\mathcal{B}$. We first construct $0=r_{1}<r_{2}<\ldots$ such that
\begin{equation}\label{tridiagerror}
\|(I-(Q_{k-1}+Q_{k}+Q_{k+1}))T_{i}Q_{k}\|\leq\frac{1}{2^{k}},
\end{equation}
for all $k,i\in\mathbb{N}$ with $k\geq i$. We will then show that for such $r_{1},r_{2},\ldots$, the conclusion of the result holds.

For $r\in\mathbb{N}$, let $P_{r}$ be the canonical projection from $\ell^{2}$ onto $\ell^{2}([1,r])$. Let $r_{1}=0$ and $P_{0}=0$. For a fixed $m\geq 2$, if $r_{1}<\ldots<r_{m-1}$ have already been chosen, then we take $r_{m}\in\mathbb{N}$ to be large enough so that $r_{m}>r_{m-1}$,
\[\|(I-P_{r_{m}})T_{i}P_{r_{m-1}}\|\leq\frac{1}{2^{m+1}},\]
and
\[\|(I-P_{r_{m}})T_{i}^{*}P_{r_{m-1}}\|\leq\frac{1}{2^{m+1}},\]
for all $i=1,\ldots,m$. By construction, these two inequalities hold for all $m,i\in\mathbb{N}$ with $i\leq m$. Thus, for given $k,i\in\mathbb{N}$ with $i\leq k$, taking $m=k$ gives
\[\|(I-P_{r_{k}})T_{i}^{*}P_{r_{k-1}}\|\leq\frac{1}{2^{k+1}},\]
and taking $m=k+2$ gives
\[\|(I-P_{r_{k+2}})T_{i}P_{r_{k+1}}\|\leq\frac{1}{2^{k+3}}.\]
So
\begin{align*}
&\|(I-(P_{r_{k+2}}-P_{r_{k-1}}))T_{i}(P_{r_{k+1}}-P_{r_{k}})\|\\\leq&
\|(I-P_{r_{k+2}})T_{i}(P_{r_{k+1}}-P_{r_{k}})\|+\|P_{r_{k-1}}T_{i}(P_{r_{k+1}}-P_{r_{k}})\|\\\leq&
\|(I-P_{r_{k+2}})T_{i}P_{r_{k+1}}\|+\|P_{r_{k-1}}T_{i}(I-P_{r_{k}})\|\leq\frac{1}{2^{k+3}}+\frac{1}{2^{k+1}}<\frac{1}{2^{k}}.
\end{align*}
This gives (\ref{tridiagerror}).

We are now ready to show that $T-\sum_{k=1}^{\infty}(Q_{k-1}+Q_{k}+Q_{k+1})TQ_{k}$ is compact for every $T\in\mathcal{B}$. First, by an argument similar to the proof of Lemma \ref{blockbound}, the series
\begin{equation}\label{gamma}
\Gamma(T):=\sum_{k=1}^{\infty}(Q_{k-1}+Q_{k}+Q_{k+1})TQ_{k}
\end{equation}
always converges in the strong operator topology for every $T\in B(\ell^{2})$. Moreover, the linear map $\Gamma:B(\ell^{2})\to B(\ell^{2})$ defined by (\ref{gamma}) is bounded. From (\ref{tridiagerror}), we have that $T_{i}-\Gamma(T_{i})$ is compact for every $i\in\mathbb{N}$. So by the denseness of the set $\{T_{1},T_{2},\ldots\}$ in $\mathcal{B}$, we conclude that $T-\Gamma(T)$ is compact for every $T\in\mathcal{B}$. Thus the result follows.
\end{proof}
\begin{proof}[Proof of Theorem \ref{mainthm}]
Let $\pi_{2}:B(\ell^{2})\to B(\ell^{2})/K(\ell^{2})$ be the quotient map. Let $U$ be the forward unilateral shift on $\ell^{2}$. Let $\mathcal{C}$ be a countable dense subset of the set of all elements of $\mathcal{A}$ that are invertible in $B(\ell^{2})/K(\ell^{2})$. Recall that the set of all invertible elements of $B(\ell^{2})/K(\ell^{2})$ with a fixed Fredholm index $k$ is path connected \cite{conway}. Thus, for every $a\in\mathcal{C}$, there is a continuous path $f_{a}:[0,1]\to B(\ell^{2})/K(\ell^{2})$ such that
\begin{enumerate}[(1)]
\item $f_{a}(0)=a$;
\item $f_{a}(1)=\pi_{2}(U)^{-\mathrm{ind}\,a}$, where $\mathrm{ind}\,a$ is the Fredholm index of $a$; and
\item $f_{a}(t)$ is invertible in $B(\ell^{2})/K(\ell^{2})$ for all $t\in[0,1]$.
\end{enumerate}
Since $f_{a}$ is continuous, $\{f_{a}(t):t\in[0,1]\}$ is separable.

Let $\mathcal{B}$ be the subalgebra of $B(\ell^{2})$ generated by $\pi_{2}^{-1}(\mathcal{A})$, $U$, $U^{*}$ and $\pi_{2}^{-1}(f_{a}(t))$ among all $a\in\mathcal{C}$ and $t\in[0,1]$. Note that $\mathcal{B}$ is separable.

Apply Lemma \ref{blocktridiag} and obtain $0=r_{1}<r_{2}<\ldots$ such that $T-\sum_{k=1}^{\infty}(Q_{k-1}+Q_{k}+Q_{k+1})TQ_{k}$ is compact for every $T\in\mathcal{B}$, where $Q_{0}=0$ and $Q_{k}$ is the canonical projection from $\ell^{2}$ onto $\ell^{2}([r_{k}+1,r_{k+1}])$. Thus, every $T\in\mathcal{B}$ is a compact perturbation of an operator in
\[\mathcal{W}_{1}=\{T\in B(\ell^{2}):\,T\ell^{2}([r_{k}+1,r_{k+1}])\subset\bigoplus_{i=k-1}^{k+1}\ell^{2}([r_{i}+1,r_{i+1}])\,\forall k\in\mathbb{N}\}.\]
Let $X_{p}=(\bigoplus_{k\in\mathbb{N}}\ell^{2}([r_{k}+1,r_{k+1}]))_{p}$. Apply Lemma \ref{mainlemma} to obtain a unital homomorphism $\Psi:\bigcup_{m\in\mathbb{N}}\mathcal{W}_{m}\to B(X_{p})$ such that
\begin{equation}\label{Psiboundess}
\frac{1}{2m+1}\|T\|_{e}\leq\|\Psi(T)\|_{e}\leq(2m+1)\|T\|_{e},
\end{equation}
for all $T\in\mathcal{W}_{m}$ and $m\in\mathbb{N}$, and that $\Psi(U^{*})$ has Fredholm index $1$.

Define a map $\Phi:\pi_{2}(\mathcal{B})\to B(X_{p})/K(X_{p})$ as follows. Let $\pi:B(X_{p})\to B(X_{p})/K(X_{p})$ be the quotient map. For a given $T\in\mathcal{W}_{1}$, define $\Phi(\pi_{2}(T))=\pi(\Psi(T))$. By (\ref{Psiboundess}), this map is a well defined unital homomorphism and
\[\frac{1}{3}\|\pi_{2}(T)\|\leq\|\Phi(\pi_{2}(T))\|\leq3\|\pi_{2}(T)\|,\]
for all $T\in\mathcal{W}_{1}$. Since every operator in $\mathcal{B}$ is a compact perturbation of an operator in $\mathcal{W}_{1}$, it follows that
\[\frac{1}{3}\|a\|\leq\|\Phi(a)\|\leq3\|a\|,\]
for all $a\in\pi_{2}(\mathcal{B})$. Moreover, $\Phi(\pi_{2}(U^{*}))=\pi(\Psi(U^{*}))$ has Fredholm index $1$.

Next we need to show that $\Phi$ preserves the Fredholm index. Let $a\in\mathcal{C}$. Observe that all the elements along the path $\Phi\circ f_{a}:[0,1]\to B(X_{p})/K(X_{p})$ are invertible in $B(X_{p})/K(X_{p})$ and thus the Fredholm index along this path is constant \cite{murphy}. In particular, $\Phi(f_{a}(0))=\Phi(a)$ and $\Phi(f_{a}(1))=\Phi(\pi_{2}(U)^{-\mathrm{ind}\,a})=\Phi(\pi_{2}(U^{*}))^{\mathrm{ind}\,a}$ have the same Fredholm index. Since we have obtained from above that $\Phi(\pi_{2}(U^{*}))$ has Fredholm index $1$, it follows that $\Phi(a)$ has Fredholm index $\mathrm{ind}\,a$, i.e., $\Phi(a)$ and $a$ have the same Fredholm index for all $a\in\mathcal{C}$. Since $\mathcal{C}$ is dense in the set of all elements of $\mathcal{A}$ that are invertible in $B(\ell^{2})/K(\ell^{2})$, it follows that $\Phi(a)$ and $a$ have the same Fredholm index for all $a\in\mathcal{A}$ that is invertible in $B(\ell^{2})/K(\ell^{2})$.

Therefore, we have constructed a unital homomorphism $\Phi:\mathcal{A}\to B(X_{p})/K(X_{p})$ preserving the Fredholm index such that $\frac{1}{3}\|a\|\leq\|\Phi(a)\|\leq3\|a\|$ for all $a\in\mathcal{A}$.

By a result of Pe{\l}czy\'nski \cite[Proposition 7]{pelczynski}, the Banach space $X_{p}$ is $C_{p}$-isomorphic to $\ell^{p}$ for some constant $C_{p}\geq 1$ that depends only on $p$. Indeed, by Rademacher functions, every finite dimensional Hilbert space can be embedded complementably into $\ell^{p}$ where both the embedding constant and the complementation constants are bounded by a fixed constant that depends only on $p$ but not the dimension. Since $X_{p}$ is a $\ell^{p}$ direct sum of finite dimensional Hilbert spaces, $X_{p}$ can be embedded complementably into $\ell^{p}$. Conversely, it is trivial that $\ell^{p}$ can be embedded complementably into $X_{p}$. With the additional fact that $\ell^{p}$ is isomorphic to a $\ell^{p}$ direct sum of infinitely many copies of itself, the argument by Pe{\l}czy\'nski gives that $X_{p}$ is $C_{p}$-isomorphic to $\ell^{p}$. This together with the conclusion from the previous paragraph proves the result.
\end{proof}

\noindent{\bf Acknowledgements:} The author is grateful to William B. Johnson for useful discussions. This paper is extracted from the author's unpublished manuscript ``Similarity of operators on $l^{p}$" arxiv 1706.08582 (2019). The author is supported by a startup fund at Michigan State University.


\begin{thebibliography}{00}
\bibitem{bp} D.~Blecher and N.~C.~Phillips, $L^{p}$-operator algebras with approximate identities I, Pacific Journal of Mathematics 303.2 (2020) 401-457.
\bibitem{isorep} M.~T.~Boedihardjo, $C^{*}$-algebras isomorphically representable on $\ell^{p}$, Analysis \& PDE (2020), 13(7), 2173-2181.
\bibitem{bdf} L.~G.~Brown, R.~G.~Douglas and P.~A.~Fillmore, Unitary equivalence modulo the compact operators and extensions of $C^{*}$-algebras, Proceedings of a Conference on Operator Theory, pp. 58-128. Lecture Notes in Math., Vol. 345, Springer, Berlin, 1973.
\bibitem{conway} J.~B.~Conway, A course in functional analysis. Vol. 96. Springer, 2019.
\bibitem{ds} N.~Dunford and J.~T.~Schwartz, Linear operators. Part III: Spectral operators, Pure and Applied Mathematics Vol. VII, With the assistance of William G. Bade and Robert G. Bartle, Interscience Publishers John Wiley \& Sons, New York-London-Sydney 1971.
\bibitem{fixman} U.~Fixman, Problems in spectral operators, Pacific J. Math. (1959) 1029-1051.
\bibitem{garthi} E.~Gardella and H.~Thiel, Extending representations of Banach algebras to their biduals, Mathematische Zeitschrift 294.3-4 (2020) 1341-1354.
\bibitem{murphy} G.~J.~Murphy, $C^{*}$-algebras and operator theory, Academic Press, Boston, 1990.
\bibitem{pelczynski} A.~Pe{\l}czy\'nski, Projections in certain Banach spaces, Studia Mathematica 2.19 (1960) 209-228.
\bibitem{NCweylvon} D.~Voiculescu, A non-commutative Weyl-von Neumann theorem, Rev. roum. math. pures et appl. 21 (1976) 97-113.
\bibitem{wang} Q.~Wang and Z.~Wang, Notes on the $\ell^{p}$-Toeplitz algebra on $\ell^{p}(\mathbb{N})$, Israel Journal of Mathematics 245.1 (2021) 153-163.
\end{thebibliography}
\end{document}